\documentclass[12pt, a4paper]{article}
\usepackage{}
\usepackage[top=2.54cm, bottom= 2.54cm, left=2.54cm, right=2.54cm]{geometry}
\def\Dj{\hbox{D\kern-.73em\raise.30ex\hbox{-}
		\raise-.30ex\hbox{}}}
\def\dj{\hbox{d\kern-.33em\raise.80ex\hbox{-}
		\raise-.80ex\hbox{\kern-.40em}}}
\usepackage{epsfig}
\usepackage{amsmath,amsthm,amsfonts,amssymb,amscd,cite}
\usepackage{mathrsfs,color}
\usepackage{latexsym}
\usepackage{graphicx}
\usepackage{color}
\usepackage{amsthm}
\allowdisplaybreaks

\newtheorem{clm}{Claim}[section]
\newtheorem{thm}{Theorem}
\newtheorem{conjecture}{Conjecture}

\newtheorem{lemma}[thm]{Lemma}

\usepackage{xcolor}

\begin{document}
	
	\vspace*{20mm}
	
	\noindent {\Large \bf The perfect divisibility and chromatic number of some odd hole-free graphs }\footnote{All authors are co-first authors with equal contributions to this paper. This work is supported by National Key Research and Development Program of China (Grants Nos.~2020YFA0712500).
			
	E-mail: hwh12@gdut.edu.cn (W. He), shiyp9@mail2.sysu.edu.cn (Y. Shi, Corresponding author), rong\_w24246@163.com (R. Wu), mcsyao@mail.sysu.edu.cn (Z. Yao).
	}
	
	\vspace{10mm}
	
	\noindent
	{\large \bf Weihua He$^a$,\, Yueping Shi$^b$,\, Rong Wu$^c$,\, Zheng-an Yao$^b$
	} \\[10mm]
	$^a${\it School of Mathematics and Statistics, Guangdong University of Technology, Guangzhou, 510520, China}\\
	$^b${\it School of Mathematics, Sun Yat-sen University, Guangzhou, 510275, China}\\
	$^c${\it School of Science, Shanghai Maritime University, Shanghai, 201306,  China}
	
	\vspace{7mm}
	
	\noindent
	%(Received November ??, 2022)
	
	\vspace{10mm}
	
	\noindent
	{\bf A B S T R A C T} \\
	A hole is an induced cycle of length at least 4, and an odd hole is a hole of odd length. It is NP-hard to color the vertices of an odd hole-free graph.  A graph $G$ is perfectly divisible if every induced subgraph $H$ of $G$ with at least one edge admits a partition of $V(H)$ into sets $A$ and $B$ such that $H[A]$ is perfect and $\omega(H[B])<\omega(H)$. $G$ is short-holed if every hole in $G$ has length 4. A hammer is the graph obtained by identifying one vertex of a $K_3$ and one end vertex of a $P_3$. 
	In this paper, we prove that (i) (odd hole, hammer, $K_{2,3}$)-free graphs are perfectly divisible, (ii) $\chi(G)\le 4\omega(G)(\omega(G)-1)$ if $G$ is short-holed and $(K_1+C_4)$-free, (iii) $\chi(G)\le 2\omega(G)-1$ if $G$ is short-holed and $(K_1\cup K_3)$-free, and (iv) $\chi(G)\le 16\omega(G)-24$ if $G$ is short-holed and $(K_1+(K_1\cup K_3))$-free.
	
	\vspace{5mm}
	
	\noindent
	{\it Classification:} 05C15, 05C17\\
	{\it Keywords:} perfect divisibility, odd hole, chromatic number\\
	
	\baselineskip=0.30in
	
	\section{Introduction}
	
	All graphs considered in this paper are finite and simple. We follow \cite{BM08} for undefined notations and terminology. Let $G$ be a graph. The {\em complement} $\overline{G}$ of $G$ is the graph with vertex set $V(G)$ and such that two vertices are adjacent in $\overline{G}$ if and only if they are not adjacent in $G$. We use $P_k$ to denote a chordless {\em path} on $k$ vertices. The {\em interior} of a path $P$ is the set of vertices of $P$ incident with two edges of $P$. The complete bipartite graph with partite sets of size $p$ and $q$ is denoted by $K_{p,q}$, and the complete graph with $l$ vertices is denoted by $K_l$. Let $v\in V(G)$, and let $X$ be a subset of $V(G)$. We say that $v$ is {\em complete} to $X$ if $v$ is adjacent to each vertex of $X$ and say that $v$ is {\em anticomplete} to $X$ if $v$ is not adjacent to any vertex of $X$. We say that $X$ is {\em complete} to $Y$ if each vertex of $X$ is complete to $Y$ and say that $X$ is {\em anticomplete} to $Y$ if each vertex of $X$ is anticomplete to $Y$. 
	
	For a vertex $v$ of a graph $G$, $N(v)$ will denote the set of vertices adjacent to $v$ (we use $N_G(v)$ if there is a risk of confusion). Two vertices $x,y$ are {\em neighbors} if $xy$ is an edge of $G$, and  {\em nonneighbors} otherwise. The closed neighborhood of $v$, denoted $N[v]$, is defined to be the set $\{v\}\cup N(v)$. We define $M(v)$ (or $M_G(v)$) to be the set $V (G)\setminus N[v]$. Given a subset $X\subseteq V(G)$, $N(X)$ (or $N_G(X)$) is the set $\{u\in V(G)\setminus X : u$ is adjacent to a vertex of $X\}$, and $M(X)$ (or $M_G(X)$) is the set $V(G)\setminus (X\cup N(X))$. The {\em union} of two vertex-disjoint graphs $G_1$ and $G_2$, denoted by $G_1\cup G_2$, is the graph with vertex set $V(G_1)\cup V(G_2)$ and edge set $E(G_1)\cup E(G_2)$, and we use $G_1+G_2$ to denote the graph with vertex set $V(G_1)\cup V(G_2)$ and edge set $E(G_1)\cup E(G_2)\cup \{xy|x\in V(G_1), y\in V(G_2)\}$.
	
	Let $u$ and $v$ be two vertices of $G$. We simply write $u\sim v$ if $uv\in E(G)$, and write $u\not\sim v$ if $uv\not\in E(G)$. We use $d(u)$ (or $d_G(u)$) to denote the degree of $u$ in $G$. The maximum degree of a graph 
	$G$, denoted by $\Delta(G)$, is defined as the maximum value among the degrees of all its vertices: $\Delta(G)=\max\{d(u):u\in V(G)\}$. Let $S$ be a subset of $V(G)$. We use $G[S]$ to denote the subgraph of $G$ induced by $S$.  For two graphs $G$ and $H$, we say that $G$ induces $H$ if $H$ is an induced subgraph of $G$, and say $G$ does not induce $H$ otherwise. Analogously, let $\cal H$ be a family of graphs, we say that $G$ is $\cal H$-free if it does not contain any induced subgraph which is isomorphic to a graph in $\cal H$. A {\em hole} of $G$ is an induced cycle of length at least 4 and a {\em $k$-hole} is a hole of length $k$. Let $C_k$ denote the hole on $k$ vertices. $G$ is {\em chordal} if $G$ is hole-free. A $k$-hole is called an {\em odd hole} if $k$ is odd, and is called an {\em even hole} otherwise. An {\em antihole} is the complement of some hole. An odd (resp. even) antihole is defined analogously. It follows from a result of \cite{KKTW01} that it is NP-hard to color an odd hole-free graph. However, it is now known that an odd hole (if it exists) can be detected in polynomial time (\cite{CSS21}).
	
	Let $k$ be a positive integer, and let $[k]=\{1, 2, \ldots, k\}$. A $k$-{\em coloring} of $G$ is a mapping $c: V(G)\mapsto [k]$ such that $c(u)\neq c(v)$ whenever $u\sim v$ in $G$. A maximal complete subgraph of a graph is a clique, and the {\em clique number} $\omega(G)$ of $G$ is the number of vertices in a largest clique of $G$. The {\em chromatic number} $\chi(G)$ of $G$ is the minimum integer $k$ such that $G$ admits a $k$-coloring. A family $\cal G$ of graphs is said to be $\chi$-bounded if there is a function $f$ such that $\chi(G)\le f(\omega(G))$ for every graph $G\in \cal G$, and if such a function $f$ does exist for $\cal G$, then $f$ is said to be a {\em binding function} of $\cal G$. For a graph $G$, we say $G$ is {\em perfectly divisible} if for each induced subgraph $H$ of $G$ with at least one edge, $V(H)$ can be partitioned into $A$ and $B$ such that $H[A]$ is perfect and $\omega(H[B])<\omega(H)$. Obviously, if $G$ is perfectly divisible, then $\chi(G)\le \omega(G)+(\omega(G)-1)+\cdots+2+1= \binom{\omega(G)+1}{2}$. 
	
	A {\em claw} is the graph $K_{1,3}$, a {\em cochair} (resp. {\em dart}) is a graph obtained from the graph $K_1+P_3$ by adding a pendant edge to a vertex of degree 2 (resp. degree 3), a {\em bull} is a graph consisting of a triangle with two disjoint pendant edges, a {\em banner} is a graph consisting of a $C_4$ with one pendant edge. A {\em hammer} is the graph obtained by identifying one vertex of a $K_3$ and one end vertex of a $P_3$. A {\em fork} is a graph obtained from $K_{1,3}$ by subdividing an edge once. An {\em odd balloon} is a graph obtained from an odd hole by identifying respectively two consecutive vertices with two leaves of $K_{1,3}$.
    
	Ho\`{a}ng and McDiarmid \cite{HM02} proved that (odd hole, claw)-free graphs are perfectly divisible. Brandst\"adt and Giakoumakis \cite{BG15} showed that (odd hole, cochair)-free graphs are perfectly divisible. Brandst\"adt and Mosca \cite{BM15} proved that odd hole-free graphs without dart or bull are perfectly divisible. Chudnovsky and Sivaraman \cite{CS19} also proved that (odd hole, bull)-free graphs are perfectly divisible. Wu and Xu \cite{WX24} showed that (odd balloon, fork)-free graphs are perfectly divisible. Ho\`{a}ng \cite{HCT18} showed that (odd hole, banner)-free graphs are perfectly divisible and proposed the following conjecture. 
    
    \begin{conjecture}[\cite{HCT18}]
    Odd hole-free graphs are  perfectly divisible.
	\end{conjecture}
    
	A set $C$ of vertices of $G$ is a {\em clique cutset} if $C$ induces a clique in $G$, and $G-C$ is disconnected.  Dong, Xu and Xu \cite{DXX22} proved that ($P_5, C_5, K_{2,3}$)-free graphs are perfectly divisible. The proof of this result relies on a theorem (A $P_5$-free minimally non-perfectly divisible graph cannot contain a clique cutset) in \cite{H25}.
	In this paper, we prove that
	
	\begin{thm}\label{hammer}
		(odd hole, hammer, $K_{2,3}$)-free graphs are perfectly divisible.
	\end{thm}
	
	Say $G$ is $\emph{short-holed}$ if every hole in $G$ has length four. In section 3, we focus on short-holed graphs.  This class of graphs is interesting because it looks simpler than the class of odd hole-free graphs, but in reality it is just as challenging as the class of odd hole-free graphs in most cases because the short induced cycles play a very important role. We may assume $\omega=\omega(G)\ge2$ since $\chi(G)\le1$ if $\omega(G)\le1$. Sivaraman \cite{SV18} proved that

\begin{thm}[\cite{SV18}]\label{T-4-1}
Let $G$ be a $short$-$holed$ graph. Then $\chi(G)\le 2^{2^{\omega}}$.
\end{thm}
	
	This result is not much better than $\chi(G)\le 2^{2^{\omega+2}}$, which is Scott and Seymour's result about odd hole-free graphs \cite{SS16}, but the proof is much simpler. Scott and Seymour \cite{SS20} also pay attention to the short-holed graph and (Scott and Seymour communicated privately with Vaidy Sivaraman) they were able to prove that $\chi(G)\le 10^{20}2^{\omega^2}$ if $G$ is a short-holed graph. However, this upper bound is still far from optimal, Scott and Seymour \cite{SS20} believe maybe $\chi(G)\le \omega^2$ for all short-holed graphs.

	A levelling $L$ is a sequence of disjoint subsets of vertices $(L_0,L_1,\ldots,L_k)$ such that $|L_0|=1$ and every vertex in $L_i$ has a neighbor in $L_{i-1}$, and no vertex in $L_i$ has a neighbor in $L_j$ for $j<i-1$. We say that a vertex $u \in L_{i-1}$ is a {\em parent} of a vertex $v \in L_i$ and $v$ is a {\em child} of $u$, if there is an edge between $u$ and $v$. 
	By using the method of Sivaraman in \cite{SV18}, which adapts the leveling argument and several ideas from Scott and Seymour \cite{SS16}, we prove that
	
	\begin{thm}\label{short hole-1}
		Let $G$ be a short-holed and $(K_1+C_4)$-free graph. Then $\chi(G)\le 4\omega(\omega-1)$.
	\end{thm}
	
	Also, we have the following results.
	
	\begin{thm}\label{short hole-2}
		Let $G$ be a short-holed and $(K_1\cup K_3)$-free graph. Then $\chi(G)\le 2\omega-1$.
	\end{thm}
	
	\begin{thm}\label{short hole-3}
		Let $G$ be a short-holed and $(K_1+(K_1\cup K_3))$-free graph. Then $\chi(G)\le 16\omega-24$.
	\end{thm}
	
	This paper is organized as follows. In section 2, we prove Theorem~\ref{hammer}. Section 3 is devoted to the proofs of Theorems~\ref{short hole-1},\ref{short hole-2} and \ref{short hole-3}.
	
	\section{Perfect divisibility of $($odd hole, hammer, $K_{2,3}$$)$-free graphs }
	The aim of this section is to prove Theorem~\ref{hammer}. 
	For an induced subgraph $H$ of a graph $G$, a vertex $v\in$ $V(G)\setminus V(H)$ that is anticomplete to $V(H)$ is called an {\em anticenter} for $H$.
	A graph $G$ is {\em minimally non-perfectly divisible} if $G$ is not perfectly divisible but each of its proper induced subgraph (with at least one edge) is. 
	
	\begin{proof}[Proof of Theorem~\ref{hammer}] Suppose to the contrary, let $G$ be an $($odd hole, hammer, $K_{2,3}$$)$-free minimally non-perfectly divisible graph. It is clear that $G$ is connected. Let $v$ be a vertex of maximum degree. $G[M(v)]=G[V(G)\setminus N[v]]$ is not perfect. (Otherwise $G$ is perfectly divisible). By the Strong Perfect Graph Theorem\cite{CRST06}, $M(v)$ contains a set $X_0$ such that $G[X_0]$ induces an odd antihole, and $v$ is anticomplete to $X_0$. Let $v_1,v_2,\ldots,v_n$($n\ge7$, $n$ is odd) be the vertices of $X_0$, with edges $v_iv_j$ whenever $|i-j|\ne 1$ (indices are modulo $n$). Let $A$ be the set of anticenters for $G[X_0]$. Since $v \in A$, $A\neq \phi$. We know each vertex of $N(A)$ has a neighbor in $X_0$. If $N(A)=\phi$, then $G=G[A]\cup G[X_0]$ is not connected, a contradiction. Thus, assume that $N(A)\neq \phi$.

	We are going to prove 
	\medskip
		
	(1) Each vertex of $N(A)$ has at most two nonneighbors in $X_0$. In particular, if some vertex of $N(A)$ has exactly two nonneighbors in $X_0$, the two nonneighbors are not adjacent in $G[X_0]$.
		
	\medskip
		
	Suppose there exists a vertex $u$ of $N(A)$ which has three nonneighbors in $X_0$. Since each vertex of $N(A)$ has a neighbor in $X_0$, $u$ has a neighbor in $X_0$, say $v_1$. 
	
	Case 1. Suppose $u$ has no neighbors in $\{v_2,v_n\}$. First, assume $u\not\sim v_i$ for some $i\in\{4,\ldots,n-2\}$. Then $\{u,v_1,v_i,v_2,v_n\}$ induces a hammer, a contradiction. So $u\sim v_i$, for all $i\in\{4,\ldots,n-2\}$. 
	Thus $u\not\sim v_3$ or $u\not\sim v_{n-1}$. By symmetry, we assume $u\not\sim v_3$. But now, $\{v,u,v_5,v_3,v_n\}$ induces a hammer, a contradiction. So $u\sim v_3$ and $u\sim v_{n-1}$. Thus, if $u\sim v_1$, then $u\sim v_2$ or $u\sim v_n$, or both.
		
	Case 2. If $u\sim v_2$ and $u\sim v_n$, then $u$ has three nonneighbors in $X_0\setminus\{v_1,v_2,v_n\}$. It is clear that there exist two adjacent vertices in the three nonneighbors, say $v_j,v_{j'}$. Since $v_1$ is complete to $X_0\setminus\{v_1,v_2,v_n\}$, $\{v,u,v_1,v_j,v_{j'}\}$ induces a hammer, a contradiction. So, $u$ has exactly one neighbor in $\{v_2,v_n\}$. 
	
	Case 3. By symmetry, we may assume $u\sim v_2$ and $u\not\sim v_n$. If there exists $v_t,t\in\{4,5,\ldots,n-2\}$ such that $u\not\sim v_t$, then $\{v,u,v_2,v_n,v_t\}$ induces a hammer, a contradiction. Therefore, $u$ is complete to $X_0\setminus\{v_3,v_{n-1},v_n\}$. Since
	$u$ has three nonneighbors in $X_0$, $u$ is anticomplete to $\{v_3,v_{n-1},v_n\}$. But now, $\{v,u,v_1,v_3,v_{n-1}\}$ induces a hammer, a contradiction.
		
	So, each vertex of $N(A)$ has at most two nonneighbors in $X_0$. Suppose $u'\in N(A)$ has exactly two nonneighbors $a,b$ in $X_0$ and $a\sim b$ in $G[X_0]$. Then there exists $v_{t'}\in X_0$ such that $G[\{a,b,v_{t'}\}]$ is a triangle. Since $u\sim v_{t'}$, $\{v,u,a,b,v_{t'}\}$ induces a hammer, a contradiction. This proves (1). 
		
	We are going to prove	
	\medskip
		
	(2) Let $N(A)=B$, then $N_B(a)$ is a clique, for each vertex $a\in A$.
		
	\medskip
		
	Suppose there exist two nonadjacent vertices $u,u'$ in $N_B(a)$, $a\in A$. By (1), there exist $v_j,v_{j+1} \in X_0$ such that $u$ and $u'$ are complete to $\{v_j,v_{j+1}\}$. But now, $\{a,u,u',v_j,v_{j+1}\}$ induces a $K_{2,3}$, a contradiction. This proves (2).
		
	\medskip
		
   We are going to prove	
   \medskip

   (3) For every component $C$ of $A$, $C$ is complete to $N_B(C)$.
   
   	\medskip
   	
	Since $G$ is connected and $V(G) =
	A\cup B\cup X_0$, $N_B(C) \neq \emptyset$.
	If $|C| = 1$, the property is immediate. Now we consider the case $|C|\geq 2$. 
	
	Let $a \in C$ and let $u \in N_C(a)$.
	We first claim that $u$ is complete to $N_B(a)$. Suppose, for contradiction, that there exists a vertex
	$w \in N_B(a)$ that is not adjacent to $u$. By (1), we can choose vertices $v_i, v_j$ from $X_0$ such that
	$\{w, v_i, v_j\}$ induces a triangle. This implies that$\{u, a, w, v_i, v_j\}$ induces a hammer, a contradiction.
	Hence, $u$ is complete to $N_B(a)$. Similarly, by interchanging the roles of $a$ and $u$ in the above
	argument, we deduce that $a$ is complete to $N_B(u)$. Repeating this argument, we conclude that
	every vertex of $C$ has the same neighborhood in $B$. Hence, $C$ is complete to $N_B(C)$. This proves (3).
	
	Recall that $v$ is a vertex of maximum degree. Let $C'$ be the component of $A$ containing $v$. By (3), we deduce that
	$N_B(v) = N_B(C') \neq \emptyset$.
	Let $v' \in N_B(v)$. By properties (1), (2), and (3), we obtain $N_G(v) \backslash \{v'\} \subsetneq N_G(v') \backslash \{v\}$, and hence $d(v) < d(v')$, contradicting the choice of $v$ as a vertex of maximum degree. This completes the proof of Theorem~\ref{hammer}.
	\end{proof}

	\section{The chromatic number of short-holed graphs}
	
	If a graph $G$ is chordal, then $G$ contains no odd hole and no odd antihole. By the Strong Perfect Graph Theorem \cite{CRST06}, every chordal graph is perfect. 	
	%\pf It is clear that $G$ is chordal, $\chi(G)=\omega$. \qed
	
	\begin{lemma}\label{L-4-2}
		Let $G$ be a short-holed and $(K_1+C_4)$-free graph. Let $S\subseteq V(G)$. If there exists a vertex $v$ in $G$ which is complete to $S$, then $\chi(G[S])\le\omega-1$.
	\end{lemma}
	\begin{proof}[Proof] It is clear that $G[S]$ is chordal and so $\chi(G[S])=\omega(G[S])\le\omega-1$. 
	\end{proof}
	
	\begin{proof}[Proof of Theorem~\ref{short hole-1}] We may assume that $G$ is connected. A levelling $L$ is a sequence of disjoint subsets of vertices $(L_0,L_1,\ldots)$ such that $|L_0|=1$ and every vertex in $L_i$ has a parent in $L_{i-1}$, and no vertex in $L_i$ has a neighbor in $L_j$ for $j<i-1$. The proof is by analyzing this leveling, in particular, looking at a level and the previous two levels. We are going to prove that 
		
	For each $k$, $\chi(L_k)\le 2\omega(\omega-1)$. 
		
	This is trivially true for $k=0$. Consider $k=1$. In fact, $L_1$ is the set of neighbors of the vertex in $L_0$, and hence $\chi(L_1)\le\omega-1$, by Lemma~\ref{L-4-2}. Now let $k\ge2$, we will show how to color the vertices in $L_k$ with $ 2\omega(\omega-1)$ colors. Since we can use the same set of colors for each component of $L_k$, we may assume that $L_k$ is connected. 
		
	For $0\leq i\leq k-1$ and for every vertex $v \in L_i$, there exists $u \in L_{i+1}$ such that $v$ is its only parent (or else we could just delete $v$). (Note that we are coloring only vertices in $L_k$, and deleting vertices in the previous levels does not change the chromatic number of $G[L_k]$.) We perform this deletion iteratively. Starting from $L_0$, we remove vertices from $L_i$ one by one until each remaining vertex $v \in L_i$ satisfies the above criterion: it must be the unique parent of some vertex $u \in L_{i+1}$. We then proceed to apply the same procedure to $L_{i+1}$, and so on.
	This process ensures that for every $0 \leq i \leq k-1$, each vertex $u \in L_{i+1}$ has a unique parent in $L_i$.
	
	For the case $k=2$, we have $\chi(L_{k-1})=\chi(L_{1}) \leq 2(\omega-1)$. Now consider $k>2$ and let $x\in L_{k-2}$. Let $y$ be a child of $x$ such that $x$ is its only parent. We partition $L_{k-1}-\{y\}$ into $A(=N_{L_{k-1}}(y))$ and $B(=L_{k-1}-A-\{y\})$. By Lemma~\ref{L-4-2}, $\chi(G[A])\le\omega-1$. Suppose there is a vertex $z\in B$ that is not adjacent to $x$. Let $a$ be a parent of $z$. Note that $a \neq x$. Now there is a path between $a$ and $x$ with interior in $L_0\cup L_1\cup \ldots\cup L_{k-3}$ or the path is $a-x$, let $P$ be a shortest such path. Also, there is a path between $y$ and $z$ with interior in $L_k$ (this is guaranteed by the connectedness of $L_k$), let $P'$ be a shortest such path. Now $a-P-x-y-P'-z-a$ is a hole of Length at least 5, which is impossible. Hence we conclude that $x$ is complete to $B$. By Lemma~\ref{L-4-2}, $\chi(G[B\cup\{y\}])=\omega(G[B])\le \omega-1$. By using different colors for $A$ and $B\cup\{y\}$, we see that $L_{k-1}$ can be colored with at most $2(\omega-1)$ colors. 
	
	Now we partition $L_k$ into sets $A_1,\ldots,A_{2(\omega-1)}$ as follows: A vertex is in $A_i$ if $i$ is the smallest of the colors of its neighbors in $L_{k-1}$.
	Since $G$ is short-holed, $G$ is odd hole-free. Since $G$ is ($K_1+C_4$)-free, $G$ is odd antihole free except $\overline{C_7}$. We will show that each $G[A_i]$ contains no $\overline{C_7}$, then $G[A_i]$ is perfect by the strong perfect graph theorem, and consequently $\chi(G[A_i])\le\omega$.

	If $v$ is colored by $i$, we denote $v$ by $v(i)$.
	\begin{clm}\label{clm-1-1}
	Consider a clique $C \subset A_i$ with vertices $c_1, \ldots, c_t$ ($t \geq 1$) such that each $c_j$ has a parent $c'_j(i)$. Then there is a vertex in $\{c'_1(i), \ldots, c'_t(i)\}$ that is adjacent to all vertices of $C$.
	\end{clm}
	We will prove the Claim by induction on the number of vertices of $C$. For the induction basis, the Claim is obviously true for $t = 1$. Consider a clique $C$ as in the Claim (with $t \geq 2$). Let $C_1 = \{c_1, \ldots, c_{t-1}\}$, $C_2 = \{c_2, \ldots, c_t\}$. By the induction hypothesis, there is a vertex $c'_p(i) \in \{c'_1(i), \ldots, c'_{t-1}(i)\}$ that is adjacent to all of $C_1$. Similarly, there is a vertex $c'_q(i) \in \{c'_2(i), \ldots, c'_t(i)\}$ that is adjacent to all of $C_2$. We may assume $c'_p(i) \not\sim c_t$ and $c'_q(i) \not\sim c_1$, for otherwise we are done. Let $Q$ be a shortest path between $c'_p(i)$ and $c'_q(i)$ with interior vertices in $L_0 \cup L_1 \cup \ldots \cup L_{k-2}$. Now, $c_1 - c_t - c'_q(i) - Q - c'_p(i) - c_1$ is a hole of length at least $5$, a contradiction. This proves Claim~\ref{clm-1-1}.

	\begin{clm}\label{clm-1-2}
		$G[A_i]$ contains no $\overline{C_7}$.
	\end{clm}
		
	Suppose to the contrary, let $C^*$ be the $\overline{C_7}$ with vertices $v_0,v_1,\ldots,v_6$ and edges $v_tv_{t+1},v_tv_{t+2}$ for each $t$ modulo 7. By Claim~\ref{clm-1-1}, three vertices of any triangle in $C^*$ have a common parent colored $i$ in $L_{k-1}$. We will prove that
	\begin{equation} \label{eqa-1-1}
	\mbox{either $v_t,v_{t+1},v_{t+2},v_{t+3}$, or $v_{t-1},v_t,v_{t+1},v_{t+2}$ have a common parent colored $i$ in $L_{k-1}$}.
	\end{equation}
			
	Suppose to the contrary, let $t=0$ be a counterexample to (\ref{eqa-1-1}). Let $w_1(i)$ be a common parent of $v_0,v_1,v_2$, $w_2(i)$ be a common parent of $v_6,v_0,v_1$ and $w_3(i)$ be a common parent of $v_1,v_2,v_3$, then $|w_1(i),w_2(i),w_3(i)|=3$. Now, $w_2(i)$ is a parent of $v_0$ but not of $v_2$, and $w_3(i)$ is a parent of $v_2$ but not of $v_0$. Note that $w_2(i) w_3(i)$ is a non-edge (both $w_2(i)$ and $w_3(i)$ received the same color in a proper coloring). Let $Q$ be a shortest path between $w_2(i)$ and $w_3(i)$ with interior in $L_0\cup L_1\cup \ldots\cup L_{k-2}$. Now $v_0-v_2-w_3(i)-Q-w_2(i)-v_0$ is a hole of length at least $5$, a contradiction. This proves (\ref{eqa-1-1}).
			
	By symmetry, $v_0,v_1,v_2,v_3$ have a common parent $w_1(i)$. Since $G$ is ($K_1+C_4$)-free, $v_5\not\sim w_1(i)$. Therefore, any common parent of $v_4,v_5,v_6$ is distinct from $w_1(i)$. Without loss of generality, by (\ref{eqa-1-1}), $v_3,v_4,v_5,v_6$ have a common parent $w_4(i)$ which is distinct from $w_1(i)$. It is clear that $v_1\not\sim w_4(i)$, otherwise $\{w_4(i),v_1,v_3,v_4,v_6\}$ induces a $K_1+C_4$.
			
	Since $v_0\sim v_5$, either $v_5\sim w_1(i)$ or $v_0\sim w_4(i)$. It follows that $v_0\sim w_4(i)$ because $v_5\not\sim w_1(i)$. Since $v_2\sim v_4$, either $v_4\sim w_1(i)$ or $v_2\sim w_4(i)$. If $v_2\sim w_4(i)$, $\{w_4(i),v_6,v_0,v_2,v_4\}$ induces a $K_1+C_4$. Therefore, $v_4\sim w_1(i)$. Since $v_6\sim v_1$, either $v_6\sim w_1(i)$ or $v_1\sim w_4(i)$. It follows that $v_6\sim w_1(i)$ because $v_1\not\sim w_4(i)$, but now $\{w_1(i),v_6,v_0,v_2,v_4\}$ induces a $K_1+C_4$, a contradiction. This proves Claim~\ref{clm-1-2}.
		
	By using different sets of colors for different $A_i$, we have $\chi(L_k)\le 2(\omega-1)\cdot\chi(G[A_i]) \le 2\omega(\omega-1)$. By using one set of colors for odd levels and another set for even levels, we conclude that $\chi(G)\le 4\omega(\omega-1)$. This proves  Theorem~\ref{short hole-1}.
	\end{proof}
	
	\begin{proof}[Proof of Theorem~\ref{short hole-2}] $\forall v\in V(G)$, $G-N[v]$ is $K_3$-free and does not contain all holes of length at least 5. Thus $G-N[v]$ is a bipartite graph. Consequently, $G-N(v)$ is a bipartite graph. Because $\omega(N(v))\le \omega(G)-1$, we deduce $\chi(G)\le 2\omega-1$ by a simple induction.
	\end{proof}
	
	\begin{proof}[Proof of Theorem~\ref{short hole-3}] Similar to the proof of Theorem~\ref{short hole-1}. We may assume that $G$ is connected. A levelling $L$ is a sequence of disjoint subsets of vertices $(L_0,L_1,\ldots)$ such that $|L_0|=1$ and every vertex in $L_i$ has a parent in $L_{i-1}$, and no vertex in $L_i$ has a neighbor in $L_j$ for $j<i-1$.  We are going to prove that 
		
		For each $k$, $\chi(L_k)\le 8\omega-12$. 
	
	This is trivially true for $k=0$. Consider $k=1$. In fact, $L_1$ is the set of neighbors of the vertex in $L_0$, and hence $\chi(L_1)\le2(\omega-1)-1=2\omega-3$ by Theorem~\ref{short hole-2}. 
	
	Assume that $L_k$ is connected. For $0\leq i\leq k-1$ and for every vertex $v \in L_i$, there exists $u \in L_{i+1}$ such that $v$ is its only parent (or else we could just delete $v$). 
	For the case $k=2$, we have $\chi(L_{k-1})=\chi(L_{1}) \leq 4\omega-6$. Now consider $k>2$ and let $x\in L_{k-2}$. Let $y$ be a child of $x$ such that $x$ is its only parent. We partition $L_{k-1}-\{y\}$ into $A(=N_{L_{k-1}}(y))$ and $B(=L_{k-1}-A-\{y\})$. Similar to the proof of Theorem~\ref{short hole-1}, $x$ is complete to $B$. Then $\chi (L_{k-1})\le \chi(G[A])+\chi(G[B\cup\{y\}])\le 4\omega-6$ by Theorem~\ref{short hole-2}. 
	
	Now we partition $L_k$ into sets $A_1,\ldots,A_{4\omega-6}$ as follows: A vertex is in $A_i$ if $i$ is the smallest of the colors of its neighbors in $L_{k-1}$. If $G[A_i]$ contains a triangle $v_1,v_2,v_3$, by the proof of claim~\ref{clm-1-1}, there exists a common parent $w(i)\in L_{k-1}$ of $v_1,v_2,v_3$. $w(i)$ has a parent $u\in L_{k-2}$. Now $\{w(i),u,v_1,v_2,v_3\}$ induces a $K_1+(K_1\cup K_3)$, a contradiction. So $G[A_i]$ is $K_3$-free. Since $G[A_i]$ is $K_3$-free and odd hole-free, $G[A_i]$ is bipartite. By using different sets of colors for different $A_i$, we conclude that $\chi(L_k)\le 8\omega-12$. 
	
	By using one set of colors for odd levels and another set for even levels, we conclude that $\chi(G)\le 16\omega-24$.
	\end{proof}
	
   We conclude with a few remarks. The relationship between the strong perfect graph theorem and odd hole-free graphs has led to this class being widely studied.  By forbidding some more subgraphs, we obtain the $\chi-$ binding function of the odd hole-free graphs.  But there is still some distance to prove Ho\`ang's conjecture completely. Further, it will be interesting to study $\chi$-binding functions for the class of (odd hole, hammer)-free graphs or (odd hole, $K_{2, 3}$)-free graphs.

\section{Acknowledgment }	
The authors are very grateful to the anonymous referees for their valuable comments and corrections, which led to an improvement of the original manuscript.

\end{document}